\documentclass[10pt,draft]{amsart}

\usepackage[utf8]{inputenc}
\usepackage[english]{babel}
\usepackage{amsmath,amssymb,amsthm,amsfonts,graphicx,amscd}
\usepackage{xcolor}
\usepackage{tikz}
 
\usetikzlibrary{arrows}
\usetikzlibrary{decorations.pathreplacing,calc,shapes.geometric,patterns}

\newtheorem{theo}{Theorem}[section]
\newtheorem{lem}[theo]{Lemma}
\newtheorem{pro}[theo]{Proposition}
\newtheorem{exa}[theo]{Example}
\newtheorem{cor}[theo]{Corollary}
\newtheorem{fact}[theo]{Fact}

\theoremstyle{definition}
\newtheorem{defi}[theo]{Definition}

\numberwithin{equation}{section}

\newcommand{\impli}{\Rightarrow}
\newcommand{\Nat}{\mathbb{N}}
\newcommand{\sub}{\subseteq}

\def\({\left(}
\def\){\right)}

\def\N{\mathbb{ N}}
\def\R{\mathbb{ R}}

\def\rto{\rightarrow}

\def\CC{\mathcal{C}}

\def\FF{\mathcal{F}}
\def\DD{\mathcal{D}}

\def\EE{\mathcal{E}}

\def\RR{\mathcal{R}}

\def\SS{\mathcal{S}}

\def\PP{\mathcal{P}}

\def\1{\textbf{1}}

\def\supp{\operatorname{supp}}
\def\co{\operatorname{co}}

\def\spn{\operatorname{span}}

\def\w{\omega}

\title{Weak$^*$-sequential properties of Johnson-Lindenstrauss spaces}

\author[A. Avil\'es]{Antonio Avil\'es}
\address{Dpto. de Matem\'{a}ticas, Facultad de Matem\'{a}ticas, Universidad de Murcia, 30100 Espinardo (Murcia), Spain}
\email{avileslo@um.es}

\author[G.\ Mart\'{\i}nez-Cervantes]{Gonzalo Mart\'{\i}nez-Cervantes}
\address{Dpto. de Matem\'{a}ticas, Facultad de Matem\'{a}ticas, Universidad de Murcia, 30100 Espinardo (Murcia), Spain}
\email{gonzalomartinezcervantes@gmail.com}

\author[J. Rodr\'{i}guez]{Jos\'e Rodr\'{i}guez}
\address{Dpto. de Ingenier\'{i}a y Tecnolog\'{i}a de Computadores,
Facultad de Inform\'{a}tica, Universidad de Murcia, 30100 Espinardo (Murcia), Spain}
\email{joserr@um.es}



\date{\today}

\thanks{Research supported by projects MTM2014-54182-P, MTM2017-86182-P (AEI/FEDER, UE)
and 19275/PI/14 (Fundaci\'on S\'eneca).}

\begin{document}

\begin{abstract}
A Banach space $X$ is said to have Efremov's property ($\mathcal{E}$) if every element of the weak$^*$-closure of a convex bounded set $C \subseteq X^*$ 
is the weak$^*$-limit of a sequence in $C$. By assuming the Continuum Hypothesis, we prove that there exist maximal almost disjoint families
of infinite subsets of $\mathbb{N}$ for which the corresponding Johnson-Lindenstrauss spaces enjoy (resp. fail) property ($\EE$). 
This is related to a gap in [A. Plichko, \emph{Three sequential properties of dual Banach spaces 
in the weak$^*$ topology}, Topology Appl. 190 (2015), 93--98] and allows to answer (consistently) questions of Plichko and Yost.
\end{abstract}

\maketitle

\section{Introduction}

A Banach space $X$ is said to have 
\begin{enumerate}
\item[(i)] {\em weak$^*$-angelic dual} if every element of the weak$^*$-closure of a bounded set $B \subseteq X^*$ 
is the weak$^*$-limit of a sequence in~$B$;
\item[(ii)] {\em Efremov's property~($\EE$)} if every element of the weak$^*$-closure of a convex bounded set $C \subseteq X^*$ 
is the weak$^*$-limit of a sequence in~$C$;
\item[(iii)] {\em Corson's property~($\CC$)} if every element of the weak$^*$-closure of a convex bounded set $C \subseteq X^*$ 
belongs to the weak$^*$-closure of a countable subset of~$C$.
\end{enumerate}
Clearly, (i)$\impli$(ii)$\impli$(iii).
Property~($\EE$) was first considered in~\cite{efr} and was studied further 
by Plichko and Yost in~\cite{pli3,pli-yos-2}. To clarify whether 
property~($\EE$) is actually different to the weak$^*$-angelicity of the dual or property~($\CC$),
they asked in~\cite[p.~352]{pli-yos-2} if  Johnson-Lindenstrauss spaces enjoy property~($\EE$).
It is well known that the Johnson-Lindenstrauss space $JL_2(\FF)$ associated to any almost disjoint family~$\FF$
of subsets of~$\mathbb{N}$ has property~($\CC$), but fails to have weak$^*$-angelic dual whenever
$\FF$ is maximal (shortly, a MAD family). Plichko~\cite{pli3} claimed that Johnson-Lindenstrauss spaces have property~($\EE$). However, 
his proof contains a gap. Under the Continuum Hypothesis (CH), we will prove the existence
of two MAD families $\FF_+$ and $\FF_-$ such that $JL_2(\FF_+)$ has property~($\EE$), while $JL_2(\FF_-)$ fails it.
We do not know whether such MAD families can be constructed in 
ZFC without any extra set-theoretic assumption. 

In particular, under CH, property~($\EE$) lies strictly between having weak$^*$-angelic dual
and property~($\CC$). We stress that other consistent examples of Banach spaces with property~($\CC$)
but not property~($\EE$) were already constructed by J.T.~Moore under~$\Diamond$ (unpublished) and 
Brech, see \cite[Section~3.3]{brech}. 

A Banach space $X$ is said to have {\em Gulisashvili's property~($\DD$)} if $\sigma(X^*)=\sigma(\Gamma)$
for any total set $\Gamma \sub X^*$, where $\sigma(\Gamma)$ denotes the $\sigma$-algebra on~$X$ generated by~$\Gamma$.
Gulisashvili~\cite{gul-J} proved that this property is enjoyed by any Banach space having weak$^*$-angelic dual
and asked whether the converse holds. One of the aims in~\cite{pli3} was to show that Johnson-Lindenstrauss spaces
separate both properties. This is indeed the case but the argument cannot rely on property~($\EE$).
To explain this we need a couple of definitions. A Banach space $X$ is said to have
\begin{itemize}
\item {\em property ($\EE'$)} if every weak$^*$-sequentially closed 
convex bounded subset of~$X^*$ is weak$^*$-closed (see \cite{gon3});
\item {\em property ($\DD'$)} if every weak$^*$-sequentially closed linear subspace of~$X^*$ is weak$^*$-closed (see \cite{pli3}).
\end{itemize}
The following diagram summarizes the relations between all these properties: 

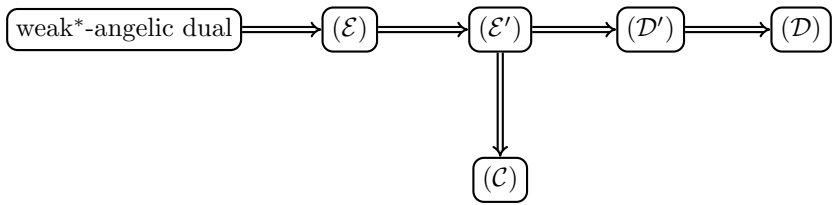
\vspace{0.1cm}
\begin{center}
\begin{tikzpicture}[thick]
\node[rectangle,rounded corners,draw=black] (nodo1) {weak$^*$-angelic dual};
\node[rectangle,rounded corners,draw=black] (nodo2) at (3,0) {($\mathcal{E}$)};
\node[rectangle,rounded corners,draw=black] (nodo3) at (5,0) {($\mathcal{E}'$)};
\node[rectangle,rounded corners,draw=black] (nodo4) at (7,0) {($\mathcal{D}'$)};
\node[rectangle,rounded corners,draw=black] (nodo5) at (9,0) {($\mathcal{D}$)};
\node[rectangle,rounded corners,draw=black] (nodo6) at (5,-2) {($\mathcal{C}$)};
\draw[-implies,double equal sign distance] (nodo1) -- (nodo2);
\draw[-implies,double equal sign distance] (nodo2) -- (nodo3);
\draw[-implies,double equal sign distance] (nodo3) -- (nodo4);
\draw[-implies,double equal sign distance] (nodo4) -- (nodo5);
\draw[-implies,double equal sign distance] (nodo3) -- (nodo6);
\end{tikzpicture}
\end{center}

\noindent The second-named author proved in~\cite{gon3} that any Johnson-Lindenstrauss space has weak$^*$-sequential dual ball and so it has property~($\EE'$).
In particular, this implies that the Johnson-Lindenstrauss space associated to any MAD family
works as a counterexample to Gulisashvili's question above.
Note also that, under CH, the space $JL_2(\FF_-)$ based on our MAD family~$\FF_-$ answers 
in the negative Plichko's question~\cite{pli3} of whether properties ($\EE$) and ($\DD'$)
are equivalent.

This paper is organized as follows. 
In Section~\ref{section:Preliminaries} we fix the terminology and collect some preliminary
facts on Johnson-Lindenstrauss spaces.
In Sections~\ref{section:JLwithE} and~\ref{section:JLwithoutE}, by assuming CH, we construct MAD families such that the corresponding 
Johnson-Lindentrauss spaces have/fail property ($\EE$). 
Finally, in Section~\ref{section:P} we analyze Plichko's attempt to prove that all Johnson-Lindentrauss spaces have property~($\EE$). 
For instance, we show that if $X$ is a Banach space which is weak$^*$-sequentially dense in~$X^{**}$, then 
$X$ has property~($\EE'$) (see Theorem~\ref{theo:seq-dense-bidual}).

\section{Preliminaries}\label{section:Preliminaries}

All our Banach spaces are real. The (topological) dual of a Banach space~$X$ is denoted by $X^*$ and the weak$^*$-topology
on~$X^*$ is denoted by~$w^*$. The linear span of a set $W \sub X$ is denoted by $\spn(W)$, while
$\overline{\spn}(W)$ stands for its closure; we write
${\rm co}(W)$ for the convex hull of~$W$. The closed unit ball of~$X$ is denoted by~$B_X$.

Two sets are said to be {\em almost disjoint} if they have finite intersection. 
By an {\em almost disjoint family} we mean a family of pairwise almost disjoint infinite subsets of~$\N$. 
An almost disjoint family is said to be a {\em maximal almost disjoint (MAD)} 
family if it is maximal with respect to inclusion.

Let $\FF$ be an almost disjoint family. The Johnson-Lindenstrauss space $JL_2(\FF)$ is defined 
as the completion of $\spn( c_0 \cup \lbrace \chi_{N}: N \in \FF \rbrace ) \sub \ell_\infty $ with respect to the norm
$$ 
	\Bigl\|x+  \sum_{r=1}^k a_r \chi_{N_r}\Big\|_{JL_2(\FF)}:= 
	\max \bigg\lbrace\Big\|x+  \sum_{r=1}^k a_r \chi_{N_{r}} \Big\|_\infty, 
	\, \bigg({\sum_{r=1}^k a_r^2}\bigg)^{\frac{1}{2}} \bigg 	\rbrace ,
	$$
where $x \in c_0$, $\{N_1,\dots,N_k\} \sub \FF$ 
and $a_1,\dots,a_k\in \R$. Here $\chi_{N}$ denotes the characteristic function of a set $N \sub \N$ and $\| \cdot \|_\infty$ is the supremum norm 
on~$\ell_\infty$. Johnson-Lindenstrauss spaces first appeared in~\cite{joh-lin} and, in general, they refer to spaces of the form 
$JL_2(\FF)$ with $\FF$ being a MAD family. However, we will avoid the maximality assumption on~$\FF$ unless otherwise mentioned.

The dual  $JL_2(\FF)^*$ is isomorphic to $\ell_1 \oplus \ell_2(\FF)$. More precisely, for each $n\in \N$, let
$e_n^*\in JL_2(\FF)^*$ be the functional satisfying 
$e_n^*(\chi_N)=\chi_N(n)$ for all $N\in \FF$ and
$$
	e_n^*(e_j)=
	\begin{cases}
	1 & \text{if $n=j$}\\
	0 & \text{otherwise}
	\end{cases}
$$
for all $j\in \N$, where $\(e_j\)_{j\in \N}$ denotes the usual basis of~$c_0$. 
For each $N \in \FF$, let $e_N^*\in JL_2(\FF)^*$ be the functional satisfying $e_N^*(x)=0$ for every $x \in c_0$ and 
$$
	e_N^*(\chi_{N'})=
	\begin{cases}
	1 & \text{if $N' = N$}\\
	0 & \text{otherwise}
	\end{cases}
$$
for all $N'\in \FF$. Then $\(e_n^*\)_{n\in \N}$ is equivalent to the usual basis of~$\ell_1$,
$\(e_N^*\)_{N\in \FF}$ is equivalent to the usual basis of~$\ell_2(\FF)$ and 
$JL_2(\FF)^*$ equals to the direct sum of $\overline{\spn}(\{e_n^*:n\in \N\})$
and $\overline{\spn}(\{e_N^*: N\in \FF\})$. That is, every $x^*\in JL_2(\FF)^*$ has a unique expression of the form 
$$
	x^*=\sum_{n \in \N} a_n e_n^* + \sum_{N \in \FF} a_N e_N^*
$$ 
with $\(a_n\)_{n\in \N} \in \ell_1$ and $\(a_N\)_{N\in \FF} \in \ell_2(\FF)$, and we write
$$
	\supp_\N x^* := \{ n \in \N: a_n \neq 0 \}
	\quad\mbox{and}\quad
	\supp_\FF x^* := \{N \in \FF: a_N\neq 0 \}.
$$  
We say that $x^*$ is \textit{finitely supported} if $\supp_\N x^*$ and $\supp_\FF x^*$ are both finite.

So, if $\FF \sub \FF'$ are two almost disjoint families, then there is an isomorphic embedding $i: JL_2(\FF)^* \rto JL_2(\FF')^*$ 
such that $i(e_n^*)=e_n^*$ for all $n\in \N$ and $i(e_N^*)=e_N^*$ for all $N\in \FF$ 
(note that $i$ is not weak$^*$-weak$^*$ continuous). This allows us 
to see every element of~$JL_2(\FF)^*$ as an element of~$JL_2(\FF')^*$ through the operator~$i$ (which will be omitted)
and we write $JL_2(\FF)^*\sub JL_2(\FF')^*$.

For more information on Johnson-Lindentrauss spaces we refer the reader to \cite{mar-pol-1}, \cite{mar-pol-3}, \cite{yos1}, \cite{yos2}
and~\cite{ziz}.

\section{A Johnson-Lindenstrauss space with property ($\EE$)}\label{section:JLwithE}

In this section we will prove that, under CH, there exists a MAD family $\FF_+$ for which $JL_2(\FF_+)$ has property~($\EE$). 
In order to do this, we construct by transfinite induction an increasing family of countable almost disjoint families $\(\FF_\alpha\)_{\alpha<\w_1}$ such that 
$\bigcup_{\alpha < \w_1} \FF_\alpha$ is a MAD family and such that every bounded sequence in $JL_2(\bigcup_{\alpha < \w_1} \FF_\alpha)^*$ 
containing~$0$ in its weak$^*$-closure is dealt with at some step to guarantee that it admits a subsequence 
whose arithmetic means are weak$^*$-null.

\begin{defi}
Given an almost disjoint family~$\FF$, we say that a 
sequence $\(x_{n}^*\)_{n \in \N}$ in $JL_2(\FF)^*$ is \textit{semi-summable} if $\sup_{j\in\N}\sum_{n \in \N} |x_n^*(e_j)| < \infty$.
\end{defi}

\begin{lem}\label{lem:summable}
Let $\FF$ be an almost disjoint family and $\(x_{n}^*\)_{n \in \N}$ a bounded
sequence in~$JL_2(\FF)^*$ for which $0$ is a weak$^*$-cluster point.
Then $\(x_{n}^*\)_{n \in \N}$ admits a semi-summable subsequence.  
\end{lem} 
\begin{proof}
Observe first that for every $x^*\in JL_2(\FF)^*$ and $c>0$ the cardinality of the set
$\{j\in \N: |x^*(e_j)|\geq c\}$ is less than or equal to $\|x^*\|c^{-1}$.
 
Assume without loss of generality
that $\|x_{n}^*\|\leq 1$ for all $n\in \N$. We next construct by induction a subsequence $\(x^*_{n_k}\)_{k\in \N}$
such that $\sum_{i=1}^k|x^*_{n_i}(e_j)|<2$ for every $k\in \N$ and $j\in \N$. Of course,
such subsequence is semi-summable.
For the first step, 
just take $x^*_{n_1}:=x^*_1$. Now, suppose $n_1<n_2<\dots<n_k$ have been already chosen
in such a way that $\sum_{i=1}^k|x^*_{n_i}(e_j)|<2$ for every $j\in \N$. Note that
$$
	S:=\Big\{j\in \N: \, \sum_{i=1}^k|x^*_{n_i}(e_j)|\geq 1\Big\} \sub
	\bigcup_{i=1}^k \Big\{j\in \N: \, |x^*_{n_i}(e_j)|\geq \frac{1}{k}\Big\}, 
$$
so $S$ is finite. Since $0$ is a weak$^*$-cluster point of $\(x_{n}^*\)_{n \in \N}$,
there exists $n_{k+1}>n_k$ such that
$$
	\sum_{i=1}^{k+1}|x^*_{n_i}(e_j)|<2
	\quad \mbox{for every }j\in S.
$$
On the other hand, given any $j\in \N\setminus S$, we have
$$
	\sum_{i=1}^{k+1}|x^*_{n_i}(e_j)|
	<1+|x^*_{n_{k+1}}(e_j)| \leq 2,
$$
because $\|x^*_{n_{k+1}}\|\leq 1$. Therefore, $\sum_{i=1}^{k+1}|x^*_{n_i}(e_j)|<2$
for every $j\in \N$. 
\end{proof}

The following result will be our key lemma in the inductive construction. 

\begin{lem}
	\label{KeyLemmaEfremov}
	Let $\FF$ be a countable almost disjoint family with $\N=\bigcup \FF$ and let $\SS$ be a countable family of semi-summable sequences 
	of finitely supported elements of~$JL_2(\FF)^*$. 
	Let $\(x_k^*\)_{k \in \N}$ be a semi-summable sequence of finitely supported elements of~$JL_2(\FF)^*$ such that
	\begin{enumerate}
	\item[$(\star)$] $\bigcup_{k\in \N}\supp_\N x_k^*$ is not contained in a finite union of elements of~$\FF$. 
	\end{enumerate}
	Then there exists an infinite set $N \sub \bigcup_{k\in \N}\supp_\N x_k^*$ such that:
	\begin{enumerate}
		\item[(i)] $\FF \cup \{N\}$ is an almost disjoint family;
		\item[(ii)] $\lim_{k\to \infty}\frac{1}{k}(x_{1}^*+x_{2}^*+\ldots+x_{k}^*)(\chi_N)=0$;
		\item[(iii)] $\lim_{k\to \infty} \frac{1}{k}(y_{1}^*+y_{2}^*+\ldots+y_{k}^*)(\chi_N)=0$ 
		for every sequence $\(y_k^*\)_{k \in \N} \in \SS$.
	\end{enumerate}
	(We consider $\chi_N \in JL_2(\FF\cup\{N\})$ and the embedding $JL_2(\FF)^* \sub JL_2(\FF\cup\{N\})^*$.)  
\end{lem}
\begin{proof}
	Enumerate $\FF= \{N_i: i \in \N\}$ and $\SS = \lbrace \(y_{r,m}^*\)_{m\in \N} : r \in \N \rbrace$. Define $A:=\bigcup_{k\in \N}\supp_\N x_k^*$. 
	For each $k\in \N$, set
	$$
		A_k:=\bigcup_{r,m \leq 2^k} \supp_\N y_{r,m}^* \cup \bigcup_{m \leq 2^k} \supp_\N x_{m}^* \cup \bigcup_{i \leq k} N_i,
	$$
	so that $A \setminus A_k$ is infinite (bear in mind $(\star)$ and the fact that $\N=\bigcup \FF$). 
	Therefore, we can take a sequence $\(t_k\)_{k\in \N}$ with $t_k \in A \setminus A_k$ for all $k\in \N$ and $t_k\neq t_{k'}$ whenever $k\neq k'$.
	Define $N := \{t_k: k \in \N \} \sub A$. Let us check that $N$ satisfies the required properties. 
		Given any $i\in \N$, we have $N_i \sub A_k$ for every $k\geq i$, hence $t_k\not\in N_i$ for every $k\geq i$, and so
	$N\cap N_i$ is finite. Therefore, $\FF\cup\{N\}$ is almost disjoint.
	Note that for any
	$$
		x^*=\sum_{n \in \N} a_n e_n^* + \sum_{M \in \FF} a_{M} e_{M}^* \in JL_2(\FF)^* \sub JL_2(\FF\cup\{N\})^*
	$$ 
	(with $\(a_n\)_{n\in \N} \in \ell_1$ and $\(a_{M}\)_{M\in \FF} \in \ell_2(\FF)$) we have 
	\begin{equation}\label{eqn:J}
		x^*(\chi_N)=\sum_{n\in N}a_n.
	\end{equation}
	
	In order to prove~(ii), fix $j \in \N$ with $j\geq 2$ and let $s(j)\in\N$ so that 
	$$
		s(j)\leq \log_2(j) < s(j)+1.
	$$ 
	Then
	$$
		N\cap \bigcup_{m \leq j} \supp_\N x_{m}^* \sub \{t_1,t_2,\dots,t_{s(j)}\}
	$$
	and so
	\begin{multline*}
		\Big|\frac{1}{j}(x_{1}^*+x_{2}^*+\ldots+x_{j}^*)(\chi_N)\Big|\stackrel{\eqref{eqn:J}}{=}
		\frac{1}{j}\Big|(x_{1}^*+x_{2}^*+\ldots+x_{j}^*)\bigg(\sum_{k=1}^{s(j)}e_{t_k}\bigg)\Big|
		\\
		\leq \frac{1}{j} \sum_{k=1}^{s(j)}\sum_{m=1}^j |x_m^*(e_{t_k})| 
		\leq \frac{s(j)}{j} C, 
	\end{multline*}
	where we write $C:=\sup_{k\in \N}\sum_{m=1}^\infty|x_m^*(e_{t_k})|<\infty$
	(bear in mind that  $\(x_m^*\)_{m \in \N}$ is semi-summable).
	It follows that $\lim_{j\to\infty}\frac{1}{j}(x_{1}^*+x_{2}^*+\ldots+x_{j}^*)(\chi_N)=0$.
	
	For the proof of~(iii), take any $r\in \N$. For every $j\in \N$ with $j\geq \max\{r,2\}$ we have
	$$
		N\cap \bigcup_{m \leq j} \supp_\N y_{r,m}^* \sub \{t_1,t_2,\dots,t_{s(j)}\}
	$$
	and, similarly as before, we conclude that 
	$\lim_{j\to\infty}\frac{1}{j}(y_{r,1}^*+y_{r,2}^*+\ldots+y_{r,j}^*)(\chi_N)=0$.	
	The proof is complete. 	
\end{proof}

\begin{theo}
	\label{theoMADfamilywithE}
	Under CH, there exists a MAD family $\FF_+$ such that every bounded sequence $\(x_n^*\)_{n\in \N}$ in $JL_2(\FF_+)^*$ 
	for which $0$ is a weak$^*$-cluster point has a subsequence $\(x^*_{n_k}\)_{k \in \N}$ such that 
	$$
		w^*-\lim_{k\to \infty}\frac{1}{k}(x_{n_1}^*+x_{n_2}^* + \ldots + x_{n_k}^*)=0.
	$$ 
\end{theo}
\begin{proof}
	Let $\{ \eta_\gamma=(\eta_{\gamma}^1,\eta_{\gamma}^2): 0<\gamma< \omega_1\}$ be an enumeration of 
	$\omega_1\times \omega_1$ with $\eta_{\gamma}^1 < \gamma$ for every $0<\gamma < \omega_1$.
	By transfinite induction on~$\alpha<\omega_1$, we will construct an increasing chain $\(\FF_\alpha\)_{\alpha<\omega_1}$
	of countable almost disjoint families and an increasing chain $\(\SS_\alpha\)_{\alpha<\omega_1}$ where each $\SS_\alpha$
	is a countable family of semi-summable sequences of finitely supported elements of~$JL_2(\FF_\alpha)^*$.
	In each step, the set $\RR_\alpha$ of all bounded sequences of finitely supported elements of~$JL_2(\FF_\alpha)^*$ has cardinality~$\mathfrak{c}=\aleph_1$
	and is enumerated as 
	$$
		\RR_\alpha=\{r_\xi: \, \xi \in \{\alpha\}\times \omega_1\}.
	$$

	Let $\FF_0$ be any countable almost disjoint family with $\N=\bigcup  \FF_0$ and let 
	$\SS_0$ be any countable family of semi-summable weak$^*$-null sequences of finitely supported elements of~$JL_2(\FF_0)^*$.
	
	Suppose now that $\(\FF_\gamma\)_{\gamma<\alpha}$ and $\(\SS_\gamma\)_{\gamma<\alpha}$ are already defined
	for some $0<\alpha<\omega_1$. Take $\hat{\FF}_\alpha:= \bigcup_{\gamma <\alpha} \FF_\gamma$ and $\hat{\SS}_\alpha:=\bigcup_{\gamma < \alpha} \SS_\gamma$.
	Note that $\hat{\FF}_\alpha$ is a {\em countable} almost disjoint family and that $\hat{\SS}_\alpha$ is a {\em countable}
	family of semi-summable sequences of finitely supported elements of $JL_2(\hat{\FF}_\alpha)$.
	
	Let $\(x_n^*\)_{n\in \N}:=r_{\eta_\alpha}\in \RR_{\eta_{\alpha}^1}$, which is already defined since $\eta_{\alpha}^1<\alpha$.
	That is, $\(x_n^*\)_{n\in \N}$ is a bounded sequence of finitely supported elements of the space $JL_2(\FF_{\eta_{\alpha}^1})^* \sub JL_2(\hat{\FF}_\alpha)$. 
	We now distinguish several cases:
	
	{\em Case~1}. If $0$ is not a weak$^*$-cluster point of $\(x_n^*\)_{n\in \N}$ 
	in $JL_2(\hat{\FF}_\alpha)^*$, then we set $\FF_\alpha:=\hat{\FF}_\alpha$ and $\SS_\alpha:=\hat{\SS}_\alpha$.
	
	{\em Case~2}. If $0$ is a weak$^*$-cluster point of $\(x_n^*\)_{n\in \N}$ in $JL_2(\hat{\FF}_\alpha)^*$, then we can take a weak$^*$-null subsequence 
	$\(x_{n_k}^*\)_{k \in \N}$, because $JL_2(\hat{\FF}_\alpha)$ is separable (bear in mind that $\hat{\FF}_\alpha$ is countable) and so 
	bounded subsets of $JL_2(\hat{\FF}_\alpha)^*$ are weak$^*$-metrizable.
	By passing to a further subsequence, not relabeled, we can assume that $\(x_{n_k}^*\)_{k \in \N}$ is semi-summable
	(apply Lemma~\ref{lem:summable}). 
	\begin{itemize}
	\item If $\bigcup_{k \in \N} \supp_\N x_{n_k}^*$ is contained in a finite union of elements of $\hat{\FF}_\alpha$, 
	then we set $\FF_\alpha:=\hat{\FF}_\alpha$ and $\SS_\alpha:=\hat{\SS}_\alpha$.
	\item If not, then we apply Lemma \ref{KeyLemmaEfremov} to $\hat{\FF}_\alpha$, $\hat{\SS}_\alpha$ and $\(x_{n_k}^*\)_{n\in \N}$ in order to obtain 
	an infinite set $N \sub \bigcup_{k \in \N} \supp_\N x_{n_k}^*$ satisfying the following conditions:
	\begin{enumerate}
		\item[(i)] $\hat{\FF}_\alpha \cup \{N\}$ is an almost disjoint family;
		\item[(ii)] $\lim_{k\to \infty}\frac{1}{k}(x_{n_1}^*+x_{n_2}^*+\ldots+x_{n_k}^*)(\chi_N)=0$;
		\item[(iii)] $\lim_{k\to \infty} \frac{1}{k}(y_{1}^*+y_{2}^*+\ldots+y_{k}^*)(\chi_N)=0$ 
		for every $\(y_k^*\)_{k \in \N} \in \hat{\SS}_\alpha$.
	\end{enumerate}
		In this case, we define $\FF_\alpha := \hat{\FF}_\alpha \cup \{N\}$ and $\SS_\alpha: =\hat{\SS}_\alpha \cup \{\(x_{n_k}^*\)_{k\in \N}\}$. 
	\end{itemize}
	
	This finishes the inductive construction. We claim that 
	$$
		\FF_+:= \bigcup_{\alpha<\omega_1} \FF_\alpha
	$$ 
	is a MAD family satisfying the required property. 
	Clearly, $\FF_+$ is almost disjoint. 
	
	For the maximality,
	take any infinite set $N'=\{n_k: k \in \N\} \sub \N$. Note that $\(e_{n_k}^*\)_{k\in \N}\in \RR_0$, hence there is
	$0<\alpha<\omega_1$ such that $\(e_{n_k}^*\)_{k\in \N}=r_{\eta_\alpha}$, that is, 
	$\(e_{n_k}^*\)_{k\in \N}$ is the sequence considered at step~$\alpha$ in the inductive construction. 
	It is easy to check that $0$ is not a weak$^*$-cluster point of $\(e_{n_k}^*\)_{k\in \N}$ in $JL_2(\hat{\FF}_\alpha)^*$
	if and only if $N'$ is contained in a finite union of elements of~$\hat{\FF}_\alpha$; in this case, there is
	$N''\in \hat{\FF}_\alpha \sub \FF_+$ such that $N'\cap N''$ is infinite.
	On the other hand, if $0$ is a weak$^*$-cluster point of $(e_{n_k}^*)_{k\in \N}$ in $JL_2(\hat{\FF}_\alpha)^*$, then 
	we find in step~$\alpha$ an infinite set $N \sub N'$ such that $\hat{\FF}_\alpha \cup \{N\}=\FF_\alpha \sub \FF_+$. This shows
	that $\FF$ is a MAD family.
	
	Let $\(x_n^*\)_{n \in \N}$ be a bounded sequence in $JL_2(\FF_+)^*$ for which $0$ is a weak$^*$-cluster point.
	We can assume without loss of generality that each $x_n^*$ is finitely supported (because finitely supported
	elements are norm-dense in $JL_2(\FF_+)^*$). Then $\(x_n^*\)_{n\in \N}$ belongs to $\bigcup_{\beta<\omega_1}\RR_\beta$
	and so there is $0<\alpha<\omega_1$ such that $\(x_n^*\)_{n\in \N}=r_{\eta_\alpha}$.
	Note that $0$ is a weak$^*$-cluster point of $\(x_n^*\)_{n \in \N}$ in $JL_2(\hat{\FF}_\alpha)^*$, so we are in Case~2 of the inductive construction
	at step~$\alpha$. Therefore, one of the following conditions holds:
	
	\textit{Condition 1: $\bigcup_{k \in \N} \supp_\N x_{n_k}^*$ is contained in a finite union of elements of $\hat{\FF}_\alpha$.} 
	In this case, we claim that $\(x_{n_k}^*\)_{k \in \N}$ is weak$^*$-null in $JL_2(\FF_+)^*$. Indeed, since
	$\(x_{n_k}^*\)_{k \in \N}$ is weak$^*$-null in $JL_2(\hat{\FF_\alpha})^*$, it suffices to show that
	$\lim_{k\to \infty}x_{n_k}^*(\chi_{N'})=0$ for every $N'\in \FF_+\setminus \hat{\FF}_\alpha$. 
	To this end, for each $k\in \N$, we write
	$$
		x_{n_k}^*=\sum_{m \in \N} a_{m,k} \, e_m^* + \sum_{N \in \hat{\FF}_\alpha} a_{N,k} \, e_N^* 
	$$
	with $\(a_{m,k}\)_{m\in \N} \in \ell_1$ and $\(a_{N,k}\)_{N\in \hat{\FF}_\alpha} \in \ell_2(\hat{\FF}_\alpha)$.
	Choose $N_1,N_2,\dots,N_p\in \hat{\FF}_\alpha$ such that $\bigcup_{k \in \N} \supp_\N x_{n_k}^* \sub \bigcup_{i=1}^p N_i$. Given any
	$N'\in \FF_+ \setminus \hat{\FF}_\alpha$, we have
	\begin{equation}\label{eqn:NoFalpha}
		x_{n_k}^*(\chi_{N'})=\sum_{m\in N'} a_{m,k}=\sum_{m\in \bigcup_{i=1}^p N'\cap N_i} a_{m,k}
		\quad
		\mbox{for all }k\in \N.
	\end{equation}
	Since $\bigcup_{i=1}^p N'\cap N_i$ is finite and $\lim_{k\to \infty}a_{m,k}=\lim_{k\to \infty}x_{n_k}^*(e_m)=0$
	for every $m\in \N$, from \eqref{eqn:NoFalpha} we get $\lim_{k\to \infty}x_{n_k}^*(\chi_{N'})=0$, as desired.
	
	\textit{Condition 2: $\bigcup_{k \in \N} \supp_\N x_{n_k}^*$ is not contained in a finite union of elements of~$\hat{\FF}_\alpha$.}
	Then
		\begin{equation}\label{eqn:JJ}
			\lim_{k\to \infty} \frac{1}{k}(x_{n_1}^*+x_{n_2}^* + \ldots + x_{n_k}^*)(\chi_{N'})=0
		\end{equation}
	for every $N' \in \FF_+ \setminus \hat{\FF}_\alpha$. Indeed, for $N'=N$ this follows from the very construction at step~$\alpha$.
	If $N'\in \FF_+ \setminus \FF_\alpha$, then $N'$ is added to~$\FF_+$ at step $\beta$ for some $\alpha<\beta <\omega_1$ and~\eqref{eqn:JJ}
	holds because $\(x_{n_k}^*\)_{k\in \N} \in \SS_\alpha \sub \hat{\SS}_\beta$.    
	
	Similarly as before, \eqref{eqn:JJ} implies that
	$\(\frac{1}{k}(x_{n_1}^*+x_{n_2}^* + \ldots + x_{n_k}^*)\)_{k\in \N}$ is weak$^*$-null in $JL_2(\FF_+)^*$.
	The proof is complete.
\end{proof}

\begin{cor}\label{cor:Efremov}
	Under CH, there exists a MAD family $\FF_+$ such that $JL_2(\FF_+)$ has property~($\EE$).
\end{cor}
\begin{proof}
	Let $\FF_+$ be the MAD family given by Theorem~\ref{theoMADfamilywithE}.	
	By linearity, it is enough to prove that if $C \sub B_{JL_2(\FF_+)^*}$ is a convex set with $0\in \overline{C}^{w^*}$, 
	then there exists a weak$^*$-null sequence contained in~$C$. This is obvious if $0\in C$, so we assume that $0\not\in C$.
	Since $(B_{JL_2(\FF_+)^*},w^*)$ is sequential (see \cite[Theorem~3.1]{gon3}), it has countable tightness and so
	there exists a sequence $\(x_n^*\)_{n\in \N}$ in $C$ with 
	\begin{equation}\label{eqn:JJJ}
		0  \in \overline{\{x_n^*: n \in \N\}}^{w^*}.
	\end{equation}
	Note that the existence of such a sequence can also be deduced from Corson's property~($\CC$) of $JL_2(\FF_+)$.
	Since $0$ is a weak$^*$-cluster point of~$\(x_n^*\)_{n\in \N}$ (by~\eqref{eqn:JJJ} and the fact that $x_n^*\neq 0$ for all $n\in \N$),
	the conclusion follows from Theorem~\ref{theoMADfamilywithE}, which ensures the existence of a subsequence 
	$\(x^*_{n_k}\)_{k \in \N}$ such that the sequence of its arithmetic means 
	$\(\frac{1}{k}(x_{n_1}^*+x_{n_2}^* + \ldots + x_{n_k}^*)\)_{k\in \N}$ (which is contained in the convex set~$C$)  
	is weak$^*$-null.
\end{proof}

\section{A Johnson-Lindenstrauss space without Property ($\EE$)}\label{section:JLwithoutE}

Given any infinite almost disjoint family~$\FF$, we have $0 \in \overline{\{e_n^*:n\in \N\}}^{w^*}$ in $JL_2(\FF)^*$, because
$w^*-\lim_{k\to \infty}e_{N_k}^*=0$ for every sequence $\(N_k\)_{k\in \N}$ of pairwise distinct elements of~$\FF$ and
$w^*-\lim_{n\in N}e_n^*=e_N^*$ for all $N\in\FF$.
In this section we will prove that, under CH, there exists a MAD family $\FF_-$ for which 
$\co(\{e_n^\ast: n\in \N \})$ does not contain weak$^*$-null sequences and consequently $JL_2(\FF_-)$ does not have property~($\EE$).
In order to construct $\FF_-$, we will focus on the matrices $\(\lambda_{i,j}\)_{i,j \in \N}$ determined by sequences 
$\(\sum_{j\in \N} \lambda_{i,j} e_j^*\)_{i\in \N}$ in $\co(\{e_n^\ast: n\in \N \})$.

\begin{defi}
We say that a matrix $\(\lambda_{i,j}\)_{i,j \in \N} \in [0,1]^{\N \times \N}$ is 
\begin{enumerate}
\item[(i)] {\em convex} if $\sum_{j \in \N } \lambda_{i,j} =1$ for every $i \in \N$; 
\item[(ii)] {\em null} if $\lim_{i\to \infty} \lambda_{i,j} =0$ for every $j \in \N$.
\end{enumerate}
\end{defi}

\begin{lem}
\label{LemmaAuxiliarADfamilies}
Let $\(N_r\)_{r\in \N}$ be a sequence of subsets of~$\N$ and let $\(\lambda_{i,j}\)_{i,j \in \N}$ be a convex null matrix. 
If $\lim_{i\to \infty} \sum_{j \in N_r} \lambda_{i,j}=0$ for every $r\in \N$, then there exists an infinite set $N' \sub \N$ 
such that $N'\cap N_r$ is finite for every $r\in \N$ and  
$$
	\limsup_{i\to \infty} \sum_{j \in N'} \lambda_{i,j} \geq \frac{1}{2}.
$$
\end{lem}
\begin{proof}
Since $\lim_{i\to \infty} \sum_{j \in N_r} \lambda_{i,j}=0$ for every $r\in \N$, 
we can find a strictly increasing sequence $\(n_r\)_{r\in \N}$ in~$\N$ such that 
$$
	\sum_{j \in N_1\cup N_2 \cup \ldots \cup N_r} \lambda_{n_r,j} \leq
	\sum_{k=1}^r\sum_{j\in N_k}\lambda_{n_r,j}< \frac{1}{2}
	\quad\mbox{for all }r\in\N.
$$
For each $r\in \N$ we have $\sum_{j \in \N} \lambda_{n_r,j} =1$ and, therefore, we can find a finite set $F_r\sub \N \setminus \( N_1\cup N_2 \cup \ldots \cup N_r\)$ 
in such a way that $\sum_{j \in F_r} \lambda_{n_r,j}\geq \frac{1}{2}$. 

Set $N' := \bigcup_{r \in \N} F_r$. Clearly, $N' \cap N_r \sub \bigcup_{s<r}F_s$ is finite for every $r \in \N$.
Notice that $\sum_{j \in N'} \lambda_{n_r,j} \geq \sum_{j \in F_r} \lambda_{n_r,j}\geq \frac{1}{2}$ for every $r\in \N$, so 
$$
	\limsup_{i\to\infty} \sum_{j \in N'} \lambda_{i,j} \geq \frac{1}{2}.
$$ 
Bearing in mind that for each finite set $F\sub \N$ we have $\lim_{i\to \infty} \sum_{j \in F} \lambda_{i,j} =0$, we conclude
that $N'$ is infinite.
\end{proof}

\begin{lem}
\label{CoroMatrices}
Let $\{\(\lambda_{i,j}^\alpha\)_{i,j\in \N}: \alpha < \w_1\}$ be a family of convex null matrices of cardinality~$\aleph_1$. Then there exists an 
almost disjoint family $\FF$ such that for every $\alpha < \w_1$ there is $N'_\alpha \in \FF$ with $\limsup_{i\to \infty} \sum_{j \in N'_\alpha} \lambda_{i,j}^\alpha >0$.
\end{lem}
\begin{proof}
Let $\FF_0$ be a countable almost disjoint family including a set $N'_0 \sub \N$ for which $\limsup_{i\to \infty} 
\sum_{j \in N'_0} \lambda_{i,j}^0 >0$. Indeed, the existence of such~$\FF_0$ follows from Lemma~\ref{LemmaAuxiliarADfamilies}
applied to an arbitrary countable infinite almost disjoint family and the convex null matrix~$\(\lambda_{i,j}^0\)_{i,j\in \N}$.

We now construct an increasing chain $\(\FF_\alpha\)_{\alpha<\omega_1}$ of countable infinite almost disjoint families by transfinite induction on~$\alpha$.
Suppose that $0<\alpha < \omega_1$ and that $\FF_\beta$ is already constructed for every $\beta<\alpha$. 
If 
\begin{equation}\label{eqn:4J}
	\lim_{i\to \infty} \sum_{j \in N} \lambda_{i,j}^\alpha=0
	\quad\mbox{for every }N \in \bigcup_{\beta<\alpha}\FF_\beta,
\end{equation}
then we can apply Lemma \ref{LemmaAuxiliarADfamilies} to $\bigcup_{\beta<\alpha}\FF_\beta$ (which is countable) and~$\(\lambda_{i,j}^\alpha\)_{i,j\in \N}$
in order to get an infinite set $N'_\alpha \sub \N$ such that $\FF_{\alpha}:=( \bigcup_{\beta<\alpha}\FF_\beta ) \cup \{N'_\alpha\}$ is
almost disjoint and 
$$
	\limsup_{i\to \infty} \sum_{j \in N'_\alpha} \lambda^\alpha_{i,j}>0.
$$ 
If \eqref{eqn:4J} fails, then we just take $\FF_{\alpha}:=\bigcup_{\beta<\alpha}\FF_\beta$.

It is clear that $\FF := \bigcup_{\beta<\w_1}\FF_\beta$ is the desired almost disjoint family.
\end{proof}

The previous lemma combined with the fact that under CH there are only $\aleph_1$-many 
convex null matrices in $[0,1]^{\N \times \N}$ provide a MAD family for which the corresponding Johnson-Lindenstrauss space does not have property ($\EE$):

\begin{theo}
Under CH, there exists a MAD family $\FF_-$ such that $JL_2(\FF_-)$ does not have property~($\EE$).
\end{theo}
\begin{proof}
Let $\FF_-$ be the almost disjoint family given by Lemma~\ref{CoroMatrices}
applied to the family of {\em all} convex null matrices (which has cardinality~$\aleph_1$ under CH).
We claim that $\FF_-$
is maximal. Indeed, if $N=\{n_k:k\in \N\} \sub \N$ is an infinite set, then we can define a convex null matrix $\(\lambda_{i,j}\)_{i,j \in \N}$ 
by the formula $\lambda_{i,j}:=1$ if $n_i=j$ and $\lambda_{i,j}:=0$ otherwise. 
Then there exists $N' \in \FF_-$ such that $\lim \sup _{i\to\infty} \sum_{j \in N'} \lambda_{i,j} \geq \frac{1}{2}$, which
clearly implies that $N \cap N' $ is infinite. Therefore, $\FF_-$ is a MAD family.

We now prove that $JL_2(\FF_-)$ does not have property~($\EE$). Let 
$$
	C:=\co(\{e_n^\ast: \, n\in \N \}) \sub JL_2(\FF_-)^*.
$$ 
As we explained at the beginning of this section, $0 \in \overline{C}^{w^*}$. 
Thus, it is enough to prove that no sequence in~$C$ is weak$^*$-convergent to zero.
Let $\(x_i^\ast\)_{i \in \N}$ be an arbitrary sequence in~$C$ 
and, for each $i \in \N$, write $x_i^\ast=\sum_{j\in \N}\lambda_{i,j}e_j^\ast$,
where $\lambda_{i,j}\geq 0$ and $\sum_{j\in \N}\lambda_{i,j}=1$ (the sum being finitely supported).
Hence 
$$
	M:=\(\lambda_{i,j}\)_{i,j\in \N}=\(x_i^\ast(e_j)\)_{i,j\in \N}\in [0,1]^{\N\times\N}
$$ 
is a convex matrix. Clearly, if $M$ is not null, then $\(x_i^\ast\)_{i \in \N}$ is not weak$^*$-null. On the other hand,
if $M$ is null, then there exists $N \in \FF_-$ such that 
$$
	\limsup_{i\to\infty} \, x_i^* (\chi_N)= \limsup_{i\to\infty} \, \sum_{j \in N} \lambda_{i,j} >0,
$$ 
so the sequence  $\(x_i^\ast\)_{i \in \N}$ cannot be weak$^*$-null either.
\end{proof}

\section{Further remarks on weak$^*$-sequential properties}\label{section:P}

Let $X$ be a Banach space. Given any set $C \sub X^*$, we denote by $S_1(C) \sub X^*$ the set of all limits of weak$^*$-convergent 
sequences contained in~$C$. Clearly, $X$ has property~($\EE$) if 
and only if $S_1(C)$ is weak$^*$-closed (equivalently, $S_1(C)=\overline{C}^{w^*}$) for every convex bounded set $C \sub X^\ast$.
The failed argument of~\cite{pli3} that all Johnson-Lindenstrauss spaces have property~($\EE$)
is based on the claim (see the proof of \cite[Proposition~8]{pli3}) that $S_1(C)$ is 
{\em norm-closed} for every convex bounded set $C \sub X^\ast$.
However, this is not always the case, as we show below.

\begin{defi}
We say that a Banach space $X$ has {\em property~($\PP$)} if $S_1(C)$ is 
norm-closed for every convex bounded set $C \sub X^\ast$.
\end{defi}

It is clear that property ($\mathcal{E}$) implies property~($\mathcal{P}$) and
that every Grothendieck space has property~($\mathcal{P}$). 

We next give an example
of a Banach space failing property~($\PP$). 
Recall that the cardinal~$\mathfrak{d}$ is defined as the least cardinality of a subset 
of~$\N^\N$ which is cofinal for the relation ``$f \leq^* g$ if and only if $f(i) \leq g(i)$ for all but finitely many $i$'s''.
One has $\aleph_1 \leq \mathfrak{d} \leq \mathfrak{c}$, but whether any of these 
are strict inequalities or equalities is independent of ZFC, see e.g.~\cite{dou2}. 

\begin{exa}\label{exa:l1Gamma}
	$\ell_1(\mathfrak{d})$ fails property ($\PP$).
\end{exa}
\begin{proof}
	For any function $f:\Nat \to \Nat$ we define $M_f:\Nat\times \Nat \to \mathbb{R}$ by
	$$
	M_f(i,j):=
	\begin{cases}
	1 & \text{if $j<f(i)$},\\
	\frac{1}{i} & \text{if $j\geq f(i)$}.
	\end{cases}
	$$
	Let $\Gamma \sub \Nat^\Nat$ be a family of functions with cardinality~$\mathfrak{d}$ which is cofinal for~$\leq^*$.
	Write $X:=\ell_1(\Gamma)$ and identify $X^*=\ell_\infty(\Gamma)$. For each $(i,j)\in \Nat \times \Nat$, define
	$x^*_{i,j}\in B_{X^*}=[-1,1]^\Gamma$ by declaring $x^*_{i,j}(f):=M_f(i,j)$ for all $f\in \Gamma$. Let $C \sub B_{X^*}$ be the convex hull of
	the $x^*_{i,j}$'s.
	
	Observe that for each $i\in \Nat$ we have $\frac{1}{i}\chi_\Gamma \in S_1(C)$, because
	$$
	w^*-\lim_{j\to \infty}x^*_{i,j} =\frac{1}{i}\chi_\Gamma.
	$$ 
	Therefore, $0 \in \overline{S_1(C)}^{\|\cdot\|}$. 
	
	We claim that $0\not \in S_1(C)$. Indeed, let $\(y_n^*\)_{n\in\N}$ be any sequence in~$C$ and, for each $n\in \Nat$, write
	$$
	y_n^*=\sum_{(i,j)\in \Nat\times\Nat}\lambda^{(n)}_{i,j} x^*_{i,j},
	$$ 
	where $\lambda^{(n)}_{i,j}\geq 0$ and $\sum_{(i,j)\in \Nat\times \Nat}\lambda^{(n)}_{i,j}=1$ (the sum being finitely supported).
	By contradiction, suppose that $\(y_n^*\)_{n\in\N}$ is weak$^*$-null. 
	
	{\sc Step~1.} Fix $i\in \Nat$. Then
	\begin{equation}\label{eqn:limitarrow}
	\lim_{n\to \infty}\sum_{j\in \Nat}\lambda^{(n)}_{i,j}=0,
	\end{equation}
	because for an arbitrary $f\in \Gamma$ and for every $n\in \Nat$ we have
	$$
	y_n^*(f)=\sum_{(k,j)\in \Nat\times\Nat}\lambda^{(n)}_{k,j} M_{f}(k,j) \geq
	\sum_{j\in \Nat}\lambda^{(n)}_{i,j} M_{f}(i,j) \geq 
	\frac{1}{i} \sum_{j\in \Nat} \lambda^{(n)}_{i,j}.
	$$
	Choose $n(i)\in \Nat$ such that
	\begin{equation}\label{eqn:arrow}
	\sum_{j\in \Nat}\lambda^{(n)}_{i,j} \leq \frac{1}{2^{i+1}}
	\quad
	\mbox{for every }n\geq n(i).
	\end{equation}
	Now we choose $\tilde{f}(i)\in \Nat$ large enough such that
	\begin{equation}\label{eqn:tails}
	\lambda^{(n)}_{i,j}=0 \quad\mbox{for every }j\geq \tilde{f}(i) \mbox{ and every }n<n(i).
	\end{equation}
	
	{\sc Step~2.}
	Pick $f\in \Gamma$ such that $\tilde{f}\leq^* f$ and fix $i_0\in \Nat$ such that $\tilde{f}(i)\leq f(i)$
	for every $i\geq i_0$. By~\eqref{eqn:limitarrow}, there is $n_0\in \Nat$ such that
	\begin{equation}\label{eqn:control}
	\sum_{i<i_0}\sum_{j\in\Nat}\lambda^{(n)}_{i,j} \leq \frac{1}{4} \quad
	\mbox{for every }n\geq n_0. 
	\end{equation}
	Observe that for each $n\in \Nat$ we have
	$$
	\sum_{i: \, n\geq n(i)} \sum_{j\in \Nat}\lambda^{(n)}_{i,j} \stackrel{\eqref{eqn:arrow}}{\leq} \sum_{i\in \Nat} \frac{1}{2^{i+1}}=\frac{1}{2}
	$$
	and so
	\begin{equation}\label{eqn:sum1}
	\sum_{i: \, n< n(i)} \sum_{j<\tilde{f}(i)}\lambda^{(n)}_{i,j} \stackrel{\eqref{eqn:tails}}{=} 
	\sum_{i: \, n< n(i)} \sum_{j\in\Nat}\lambda^{(n)}_{i,j}
	=
	1- \sum_{i: \, n\geq n(i)} \sum_{j\in\Nat}\lambda^{(n)}_{i,j} \geq \frac{1}{2}.
	\end{equation}
	Therefore, for every $n\geq n_0$ we have
	\begin{multline*}
	y_n^*(f)=
	\sum_{(i,j)\in \Nat\times \Nat}\lambda^{(n)}_{i,j}M_f(i,j) 
	\geq
	\sum_{\substack{i: \, n< n(i) \\ i\geq i_0}} \sum_{j< f(i)}\lambda^{(n)}_{i,j}
	\\ 
	\geq
	\sum_{\substack{i: \, n< n(i) \\ i\geq i_0}} \sum_{j< \tilde{f}(i)}\lambda^{(n)}_{i,j} 
	\stackrel{\eqref{eqn:control}}{\geq} 
	\sum_{i: \, n< n(i)} \sum_{j< \tilde{f}(i)}\lambda^{(n)}_{i,j}-\frac{1}{4}
	\stackrel{\eqref{eqn:sum1}}{\geq} \frac{1}{4},
	\end{multline*}
	which contradicts the fact that $\(y_n^*\)_{n\in\N}$ is weak$^*$-null.
\end{proof}

The remainder of this section is devoted to showing how the techniques of~\cite{pli3}
yield Theorem~\ref{theo:seq-dense-bidual} below. 
Recall first that a Banach space $X$ is said to have {\em property~($\EE'$)} (see~\cite{gon3}) if 
$S_{\omega_1}(C)=\overline{C}^{w^*}$ for every convex bounded set $C \sub X^*$. 
Here, for any ordinal $\alpha\leq \omega_1$, the $\alpha$-th $w^*$-sequential closure of a set $D \sub X^*$ is defined by transfinite induction 
as follows: $S_0(D):=D$, $S_{\alpha}(D):=S_1(S_\beta(D))$ if $\alpha=\beta+1$ and $S_{\alpha}(D):=\bigcup_{\beta<\alpha}S_\beta(D)$
if $\alpha$ is a limit ordinal.

\begin{theo}\label{theo:seq-dense-bidual}
	Let $X$ be a Banach space which is weak$^*$-sequentially dense in~$X^{**}$. Then: 
	\begin{enumerate}
	\item[(i)] $\overline{C}^{w^*}=\overline{S_1(C)}^{\|\cdot\|}=S_2(C)$ for every convex bounded set $C \sub X^*$; 
	\item[(ii)] $X$ has property ($\EE'$);
	\item[(iii)] $X$ has property~($\EE$) if and only if it has property~($\PP$).
	\end{enumerate}
\end{theo}

In order to prove Theorem~\ref{theo:seq-dense-bidual} we need some previous work. The following fact
will play a key role in our argument (see \cite[Theorem~4]{mus6}, cf. \cite[Proposition~3.9]{rod-ver}):
 
\begin{fact}\label{fact:l1}
If $X$ is a Banach space which is weak$^*$-sequentially dense in~$X^{**}$, then
$X$ contains no isomorphic copy of~$\ell_1$.
\end{fact}
 
The next result goes back to \cite[Theorem~3.3.8]{efr}, cf. \cite[Proposition~7]{pli3}. We provide
another proof for the reader's convenience. Recall that, given a Banach space~$X$
and a weak$^*$-compact set $K \sub X^*$, a set $B \sub K$ is said to be a {\em James boundary} of~$K$ 
if for every $x \in X$ there is $x_0^* \in B$ such that $x_0^*(x)=\sup_{x^* \in K} x^*(x)$. 

\begin{pro}\label{pro:CMO}
	Let $X$ be a Banach space which is weak$^*$-sequentially dense in~$X^{**}$. Then  
	for every weak$^*$-compact set $K \sub X^*$ and every James boundary $B$ of~$K$ we have
	$$
		\overline{{\rm co}(K)}^{w^*}=\overline{{\rm co}(B)}^{\|\cdot\|}.
	$$
\end{pro}
\begin{proof} The inclusion $\overline{{\rm co}(K)}^{w^*}\supseteq \overline{{\rm co}(B)}^{\|\cdot\|}$ is obvious.
Since $X$ contains no isomorphic copy of~$\ell_1$ (Fact~\ref{fact:l1}), we can apply \cite[Theorem~5.4]{cas-mun-ori} to get 
$$
	\overline{{\rm co}(K)}^{w^*}=\overline{{\rm co}(B)}^{\gamma},
$$ 
where $\gamma$ denotes the topology on~$X^*$ of uniform convergence on bounded countable subsets of~$X$.
Since $X$ is sequentially weak$^*$-dense in~$X^{**}$, it is easy to check that
$\gamma$ is stronger than the weak topology of~$X^*$, hence
$$
	\overline{{\rm co}(K)}^{w^*} = \overline{{\rm co}(B)}^{\gamma} \sub \overline{{\rm co}(B)}^{\text{weak}}=
	\overline{{\rm co}(B)}^{\|\cdot\|},
$$
which finishes the proof.
\end{proof}

The dual ball $B_{X^*}$ of a Banach space~$X$ is said to be {\em convex block weak$^*$-compact} if every 
sequence in~$B_{X^*}$ admits a weak$^*$-convergent convex block subsequence. 
By a {\em convex block subsequence} of a sequence $\(f_n\)_{n\in \N}$ in a linear space 
we mean a sequence $\(g_k\)_{k\in \N}$ of vectors of the form
$g_k=\sum_{n\in I_k}a_n f_n$,
where $\(I_k\)_{k\in\N}$ is a sequence of finite subsets of~$\N$ with $\max(I_k) < \min(I_{k+1})$ and $\(a_n\)_{n\in \N}$ is a sequence
of non-negative real numbers such that $\sum_{n\in I_k}a_n=1$ for all~$k \in \N$. 

\begin{lem}\label{lem:boundary}
	Let $X$ be a Banach space such that $B_{X^*}$ is convex block weak$^*$-compact.
	Let $D\sub X^*$ be a bounded set and $x\in X$. Then there is $x_0^*\in S_1({\rm co}(D))$ such that
	\begin{equation}\label{eqn:boundary}
		x_0^*(x)=\sup\big\{x^*(x): \, x^*\in \overline{D}^{w^*}\big\}.
	\end{equation}
\end{lem}
\begin{proof}
	Let $\alpha$ be the right hand side of~\eqref{eqn:boundary}. Take a sequence $\(x_n^*\)_{n\in\N}$ in~$D$
	such that $\lim_{n\to \infty} x_n^*(x) = \alpha$. By the assumption on~$X$, there is a weak$^*$-convergent 
	convex block subsequence~$\(y^*_k\)_{k\in\N}$ of~$\(x_n^*\)_{n\in \N}$,
	with limit $x_0^*\in S_1({\rm co}(D))$. Since $\(y_k^*(x)\)_{k\in \N}$ is a convex block subsequence of~$\(x_n^*(x)\)_{n\in\N}$, we have
	$x_0^*(x)=\alpha$.  
\end{proof}

\begin{proof}[Proof of Theorem~\ref{theo:seq-dense-bidual}]
Since $X$ contains no isomorphic copy of~$\ell_1$ (Fact~\ref{fact:l1}), $B_{X^*}$ is convex block weak$^*$-compact,
see~\cite[Proposition~3.11]{bou:79} (cf. \cite[Proposition~11]{pfi-J} and~\cite{sch-3}). 
To prove~(i), let $C \sub X^*$ be a convex bounded set. By Lemma~\ref{lem:boundary}, the convex set $S_1(C)$ is a James boundary of
the convex weak$^*$-compact set $\overline{C}^{w^*}$. 
Now, an appeal to Proposition~\ref{pro:CMO} ensures that 
$\overline{C}^{w^*}=\overline{S_1(C)}^{\|\cdot\|}$. 
Since $\overline{S_1(C)}^{\|\cdot\|}$ is (obviously) contained in~$S_2(C) \sub \overline{C}^{w^*}$, we conclude that
$\overline{C}^{w^*}=\overline{S_1(C)}^{\|\cdot\|}=S_2(C)$. 
Statements (ii) and~(iii) follow at once from~(i).
\end{proof}

It is known that $JL_2(\FF)$ is sequentially weak$^*$-dense in its bidual for every almost disjoint family~$\FF$. 
Therefore, under CH, the property of being sequentially weak$^*$-dense in the bidual does not imply property~($\PP$).

\subsection*{Acknowledgements}
This research was supported by projects MTM2014-54182-P, MTM2017-86182-P (AEI/FEDER, UE) and 19275/PI/14 (Fundaci\'on S\'eneca).

\def\cprime{$'$}
\providecommand{\bysame}{\leavevmode\hbox to3em{\hrulefill}\thinspace}
\providecommand{\MR}{\relax\ifhmode\unskip\space\fi MR }
\providecommand{\MRhref}[2]{%
  \href{http://www.ams.org/mathscinet-getitem?mr=#1}{#2}
}
\providecommand{\href}[2]{#2}


\begin{thebibliography}{10}

\bibitem{bou:79}
J.~Bourgain, \emph{La propri\'{e}t\'{e} de {R}adon-{N}ikod\'{y}m}, Publ. Math.
  Univ. Pierre et Marie Curie \textbf{36} (1979).

\bibitem{brech}
C.~Brech, \emph{{Constru\c{c}\~{o}es gen\'ericas de espa\c{c}os de Asplund
  C(K)}}, Universidade de S\~{a}o Paulo and Universit\'e Paris 7 (2008).

\bibitem{cas-mun-ori}
B.~Cascales, M.~Mu{\~n}oz, and J.~Orihuela, \emph{James boundaries and
  {$\sigma$}-fragmented selectors}, Studia Math. \textbf{188} (2008), no.~2,
  97--122. \MR{2430997}

\bibitem{efr}
N.~M. Efremov, \emph{{D}uality conditions of a {B}anach space in terms of norm
  attaining linear functionals}, Ph.D. Thesis, Kharkiv State University, 1985
  (in Russian).

\bibitem{gul-J}
A.~B. Gulisashvili, \emph{Estimates for the {P}ettis integral in interpolation
  spaces, and a generalization of some imbedding theorems}, Soviet Math., Dokl.
  \textbf{25} (1982), 428--432. \MR{0651235}

\bibitem{joh-lin}
W.~B. Johnson and J.~Lindenstrauss, \emph{Some remarks on weakly compactly
  generated {B}anach spaces}, Israel J. Math. \textbf{17} (1974), 219--230.
  \MR{0417760}

\bibitem{mar-pol-1}
W.~Marciszewski and R.~Pol, \emph{On {B}anach spaces whose norm-open sets are
  {$F\sb \sigma$}-sets in the weak topology}, J. Math. Anal. Appl. \textbf{350}
  (2009), no.~2, 708--722. \MR{2474806 (2010a:46060)}

\bibitem{mar-pol-3}
W.~Marciszewski and R.~Pol, \emph{On {B}orel almost disjoint families}, Monatsh. Math.
  \textbf{168} (2012), no.~3-4, 545--562. \MR{2993963}

\bibitem{gon3}
G.~Mart\'inez-Cervantes, \emph{Banach spaces with weak*-sequential dual ball},
  Proc. Amer. Math. Soc. \textbf{146} (2018), no.~4, 1825--1832. \MR{3754364}

\bibitem{mus6}
K.~Musia{\l}, \emph{The weak {R}adon-{N}ikod\'ym property in {B}anach spaces},
  Studia Math. \textbf{64} (1979), no.~2, 151--173. \MR{537118}

\bibitem{pfi-J}
H.~Pfitzner, \emph{Boundaries for {B}anach spaces determine weak compactness},
  Invent. Math. \textbf{182} (2010), no.~3, 585--604. \MR{2737706}

\bibitem{pli3}
A.~Plichko, \emph{Three sequential properties of dual {B}anach spaces in the
  weak{$\sp *$} topology}, Topology Appl. \textbf{190} (2015), 93--98.
  \MR{3349508}

\bibitem{pli-yos-2}
A.~M. Plichko and D.~Yost, \emph{Complemented and uncomplemented subspaces of
  {B}anach spaces}, Extracta Math. \textbf{15} (2000), no.~2, 335--371. \MR{1823896}

\bibitem{rod-ver}
J.~Rodr{\'{\i}}guez and G.~Vera, \emph{Uniqueness of measure extensions in
  {B}anach spaces}, Studia Math. \textbf{175} (2006), no.~2, 139--155.
  \MR{2261730}

\bibitem{sch-3}
Th. Schlumprecht, \emph{On dual spaces with bounded sequences without weak{$\sp
  *$}-convergent convex blocks}, Proc. Amer. Math. Soc. \textbf{107} (1989),
  no.~2, 395--408. \MR{979052}

\bibitem{dou2}
E.~K. van Douwen, \emph{The integers and topology}, Handbook of set-theoretic
  topology, North-Holland, Amsterdam, 1984, pp.~111--167. \MR{776622}

\bibitem{yos1}
D.~Yost, \emph{The {J}ohnson-{L}indenstrauss space}, Extracta Math. \textbf{12}
  (1997), no.~2, 185--192. \MR{1607165}

\bibitem{yos2}
D.~Yost, \emph{A different {J}ohnson-{L}indenstrauss space}, New Zealand J.
  Math. \textbf{37} (2008), 47--49. \MR{2406564}

\bibitem{ziz}
V.~Zizler, \emph{Nonseparable {B}anach spaces}, Handbook of the geometry of
  Banach spaces, Vol.\ 2, North-Holland, Amsterdam, 2003, pp.~1743--1816.
  \MR{1999608}

\end{thebibliography}
\end{document}